\documentclass[reqno, 11pt]{amsart}
\makeatletter
\let\origsection=\section \def\section{\@ifstar{\origsection*}{\mysection}}
\def\mysection{\@startsection{section}{1}\z@{.7\linespacing\@plus\linespacing}{.5\linespacing}{\normalfont\scshape\centering\S}}
\makeatother
\usepackage{amsmath,amssymb,amsthm}
\usepackage{mathrsfs}
\usepackage{dsfont}
\usepackage[babel]{microtype}
\usepackage[dvipsnames]{xcolor}
\usepackage[backref]{hyperref}
\hypersetup{
	colorlinks,
    linkcolor={red!60!black},
    citecolor={green!60!black},
    urlcolor={blue!60!black},
}

\usepackage{bookmark}
\usepackage{subcaption}
\usepackage{graphicx}
\graphicspath{{../figures/}}
\usepackage{verbatim}

\usepackage{bbm}
\usepackage[normalem]{ulem} 
\usepackage{standalone}
\usepackage[shortlabels]{enumitem}

\usepackage[abbrev,msc-links,backrefs]{amsrefs}
\usepackage{doi}

\renewcommand{\PrintDOI}[1]{\doi{#1}}

\usepackage[T1]{fontenc}
\usepackage{lmodern}

\usepackage[english]{babel}
\numberwithin{equation}{section}

\usepackage{geometry}
\let\polishlcross=\l
\def\l{\ifmmode\ell\else\polishlcross\fi}

\newcommand{\calF}{\mathcal{F}}

\newcommand{\calP}{\mathcal{P}}

\newcommand{\Oh}{\mathrm{O}}

\usepackage{enumitem}

\newtheorem{theorem}{Theorem}
\newtheorem{conjecture}[theorem]{Conjecture}
\newtheorem{fact}[theorem]{Fact}

\newtheorem{lemma}[theorem]{Lemma}
\newtheorem{proposition}[theorem]{Proposition}

\newtheorem{question}[theorem]{Question}
\newtheorem{corollary}[theorem]{Corollary}

\let\eps=\varepsilon
\let\theta=\vartheta
\let\phi=\varphi

\def\NN{\mathds N}

\usepackage{centernot}
\usepackage{datetime}
\usepackage{lineno}

\newcommand{\ceil}[1]{{\left\lceil #1 \right\rceil}}

\begin{document}

\title{Packing large balanced trees into bipartite graphs}

\author[C. G. Fernandes]{Cristina G. Fernandes}
\address{Instituto de Matem\'atica e Estat\'{\i}stica, Universidade de S\~ao
Paulo, S\~ao Paulo, Brazil}
\email{cris@ime.usp.br}

\author[T. Naia]{Tássio Naia}
\address{Centre de Recerca Matem\`atica, Belaterra, Spain}
\email{tnaia@member.fsf.org}

\author[G. Santos]{Giovanne Santos}
\address{Universidad de Chile, Chile}
\email{gsantos@dim.uchile.cl}

\author[M. Stein]{Maya Stein}
\address{Universidad de Chile, Chile}
\email{mstein@dim.uchile.cl}

\thanks{
  This research was initiated at the first
``ChiPaGra: Workshop Chileno Paulista em/en Grafos'', and was partially supported by
  the research project  2019/13364-7: ``Extremal and Structural Problems in Graph Theory'',
  funded by FAPESP (The São Paulo Research Foundation)
  and  ANID (Agencia Nacional de Investigación y
  Desarrollo). C.~G. Fernandes acknowledges partial
  support by CNPq (Proc.~423833/2018-9 and~310979/2020-0). G. Santos
  was supported by ANID/Doctorado Nacional/21221049.
  M. Stein was supported by ANID Regular Grant 1221905 and by ANID grant CMM Basal FB210005.
  T.~Naia was supported by the Grant PID2020-113082GB-I00
  funded by MICIU/AEI/10.13039/501100011033.
  This work is supported by the Spanish State Research Agency,
  through the Severo Ochoa and María de Maeztu Program for Centers
  and Units of Excellence in R\&D (CEX2020-001084-M)
}

\begin{abstract}
We prove that for every~${\gamma > 0}$
  there exists~$n_0 \in \NN$ such that for every~${n \geq n_0}$
  any family of up to $\lfloor{n^{\frac12+\gamma}}\rfloor$ trees having at most
  $(1-\gamma)n$ vertices in each bipartition class can be packed into~$K_{n,n}$.
  As a tool for our proof, we show an approximate bipartite version of the
  Koml\'os--S\'ark\"ozy--Szemer\'edi Theorem, which we believe to be of independent interest.
\end{abstract}


\maketitle

\section{Introduction}

A family $\mathcal H$ of graphs 
\emph{packs}
into a graph $G$
if $G$ contains pairwise edge-disjoint copies of the members of $\mathcal H$,
i.e.~a {\it packing} of $\mathcal H$. A packing is  \emph{perfect} if it uses  each edge of $G$.
A graph $H$ \emph{decomposes} a graph $G$ if there is
a perfect packing of copies of~$H$ into $G$.

Packing and decomposition
problems have a long history, naturally relating to
Euler's 1782 question  on the existence of orthogonal latin squares,
and to the existence of
designs. 
We are interested in packing trees
into complete bipartite graphs.

\subsection{Ringel-type conjectures}

In 1963, Ringel~\cite{R63-theory_of_graphs} conjectured that a perfect packing of any tree into $K_m$ for an appropriate $m$ should always exist, where $K_m$ is the complete graph on $m$ vertices.

\begin{conjecture}[Ringel~\cite{R63-theory_of_graphs}]
  \label{conj:ringel}
 Every tree
of order $n+1$ decomposes~$K_{2n+1}$.
\end{conjecture}

Ringel's
conjecture was solved in 2021 for large~$n$ by Montgomery, Pokrovskiy, and
Sudakov~\cite{MPS21-ringel_conjecture}, with a different proof given by
Keevash and Staden~\cite{KS20-ringel_conjecture}.

The following generalization of  Conjecture~\ref{conj:ringel},
 suggested in 2016, has recently been confirmed
for large trees with maximum degree $\Oh(n/\log n)$ in~\cite{ABCHPT22-tree_packing}.

\begin{conjecture}[B\"{o}ttcher, Hladký, Piguet and Taraz~\cite{BHPT16-approximate_tpc}]
  \label{conj:ringel_any_family}
  Each family  of trees of individual
  orders at most $n+1$ and total number of edges at most $\binom{2n+1}{2}$ packs into~$K_{2n+1}$.
\end{conjecture}

Much earlier, in 1989, Graham and H\"{a}ggkvist~\cite{H89-decompositions_bipartite}
generalized Ringel's conjecture to regular graphs. They
conjectured that every tree of order~$n+1$ decomposes any~{$2n$-regular} graph (which would imply Conjecture~\ref{conj:ringel}), and that
 every tree of order~$n+1$ decomposes any $n$-regular bipartite graph.
  The latter conjecture would imply the following for complete bipartite hosts.

\begin{conjecture}[Graham and H\"{a}ggkvist~\cite{H89-decompositions_bipartite}]
  \label{conj:ringel_bipartite}
 Any tree of order $n+1$ decomposes the complete bipartite graph~$K_{n,n}$.
\end{conjecture}

Evidence for the conjectures
from~\cite{H89-decompositions_bipartite} was given in~\cite{H89-decompositions_bipartite, LL05-decompositions_knn, DL13-almost_every_tree}, and a generalization in~\cite{Wang09-packing}.
Recently, Matthes~\cite{M23-large_rainbow} showed an approximate version of Conjecture~\ref{conj:ringel_bipartite}
by
proving that, for every $\eps > 0$ and sufficiently large $n$, any tree
of order $n+1$ packs at least $n$ times into $K_{(1+\eps)n,(1+\eps)n}$.

We remark that although in~\cite{H89-decompositions_bipartite}  a variant of Conjecture~\ref{conj:ringel_any_family} for regular graphs is conjectured, the analogue for bipartite graphs is not true. More precisely,
 in Conjecture~\ref{conj:ringel_bipartite}, we cannot replace the decomposition into copies of a fixed tree with a decomposition into any given family of trees of order $n+1$ each,
 since for instance the family consisting of one path and $n-1$ stars does not decompose $K_{n,n}$ if $n\ge 3$.

\subsection{Gyárfás' Tree Packing Conjecture}

In 1976, Gyárfás (see~\cite{GL76-trees_into_kn}) put forward
a conjecture which is now widely known as the Tree Packing Conjecture.
For ease of notation, let $T_i$  denote an arbitrary tree of
order~$i$.

\begin{conjecture}[Gyárfás' Tree Packing Conjecture~\cite{GL76-trees_into_kn}]\label{tpc}
  Any family of trees $T_1,\dots,T_n$ packs into~$K_n$.
\end{conjecture}

 Interestingly,
Gyárfás~\cite{G14-packing_chromatic} proved that Conjecture~\ref{tpc} is equivalent to a reformulation where $K_n$ is replaced with an arbitrary $n$-chromatic graph, which had been conjectured in~\cite{GKP12-generalizations_tpc}.
 Partial results on Conjecture~\ref{tpc} are given in~\cite{D98-almost_stars, D02-complete, D07-bounded_diameter, F83-packing_graphs, GKP12-generalizations_tpc, HBK87-trees_complete_graphs, S79-packing_trees}, and (asymptotic versions of) extensions of the conjecture can be found in~\cite{ABCHPT22-tree_packing, ABHP19-packing_degenerate, BHPT16-approximate_tpc, FLM17-packing_separable, KKOT19-blow_up, MRS16-packing_minor}.
Three results are of particular interest in our context:  Bollobás~\cite{B83-packing_trees}  showed that one can pack  $T_1,\dots,T_s$
into $K_n$ if~$s < \frac 1{\sqrt{2}}\,n$. Balogh
and Palmer~\cite{BP13-packing_conjecture} proved that, for sufficiently large~$n$, one can pack $T_n,\dots,T_{n-t+1}$ into~$K_n$ if~$t\le\frac{1}{4}\,n^{1/4}$ and none of the
$T_i$ is a star; or if
$t\le\frac{1}{4}n^{1/3}$ and
 the maximum degree of each $T_i$ is at least~$2n^{1/3}$.
 Janzer and Montgomery~\cite{JM24-largest_trees}
 improved this by showing that the~$\Omega(n)$ largest trees of any
 sequence~$T_1,\dots,T_n$ packs into~$K_n$.

In~\cite{HBK87-trees_complete_graphs} a
variant of Conjecture~\ref{tpc} for bipartite host graphs was  suggested.

\begin{conjecture}[Hobbs, Bourgeois, and Kasiraj~\cite{HBK87-trees_complete_graphs}]
  \label{conj:tpc_bipartite}
  Any family of trees $T_1,\dots,T_n$ packs into~$K_{\ceil{\frac{n}{2}},n-1}$.
\end{conjecture}

After previous work in~\cite{CR90-bipartite_conjecture, ZL77-decomposition_graphs}, Yuster~\cite{Y95-complete_bipartite}
showed that if $s \leq \left\lfloor\sqrt{\frac58}\,n\right\rfloor \approx 0.79 n$,
then any family of
trees~$T_1,\dots,T_s$ packs into~$K_{\frac{n}{2},n-1}$, for~$n$ even, and
into~$K_{\frac{n-1}{2},n}$, for~$n$ odd.
Kim, K\"uhn, Osthus, and Tyomkyn~\cite[Corollary 8.6]{KKOT19-blow_up} gave a packing of almost all trees~$T_1,\dots,T_n$ (leaving out the first few trees), if their degree is bounded by a constant, and $n$ is large.
B\"{o}ttcher, Hladký, Piguet, and Taraz~\cite[Theorem 44]{BHPT16-approximate_tpc} proved a version of this result for families of trees whose total number of edges is at most~$e(K_{\ceil{\frac{n}{2}},n-1})$.

%
%
%
%

In 2013, Hollingsworth~\cite{H13-balanced_trees} presented a variant of
Conjecture~\ref{tpc} for balanced trees.
Let~$T_{i,i}$ denote an arbitrary tree with $i$ vertices in each partition
class.

\begin{conjecture}[Hollingsworth~\cite{H13-balanced_trees}]
  \label{conj:tpc_balanced}
  Any family of trees $T_{1,1},\dots,T_{n,n}$ packs into $K_{n,n}$.
\end{conjecture}

Using Yuster's~\cite{Y95-complete_bipartite} result,
Hollingsworth~\cite{H13-balanced_trees} proved\footnote{A slightly weaker bound for $s$ is stated in~\cite{H13-balanced_trees}, but one can infer the bound we state here from the proof in~\cite{H13-balanced_trees}.} that
 $T_{1,1},\dots,T_{s,s}$ pack into $K_{n,n}$, for $n \geq 3$ and
$s \leq \frac 1{\sqrt{2}}\,n$.

\subsection{Ringel-type tree-packing for balanced trees}

Our first result, Theorem~\ref{thm:packing_trees} below,
is a step towards a variant of Ringel's conjecture for bipartite graphs,
in the lines of Conjecture~\ref{conj:tpc_balanced},
where we manage to pack trees of roughly the same order as the host graph's.

\begin{theorem}\label{thm:packing_trees}
  For every $\gamma > 0$, there exists $n_0$ such that for every~${n \geq n_0}$
  any family of at most~$n^{\frac 12-\gamma}$ rooted trees, each with at most
  $(1-\gamma)n$ vertices in either partition class,
  packs into~$K_{n,n}$, with all tree roots embedded in
  the same part of~$K_{n,n}$.
\end{theorem}

Note that the packing from Theorem~\ref{thm:packing_trees} is only possible because we are far from fully decomposing~$K_{n,n}$. Indeed, $T_{n,n}$ cannot possibly decompose $K_{n,n}$. Furthermore, for every $\eps>0$, there exists a balanced tree $T_{(1-\eps)n,(1-\eps)n}$ which does not approximately decompose~$K_{n,n}$. To see this, consider a balanced double star $D$ with $(1-\eps)n$ vertices in either partition class and note that any vertex of $K_{n,n}$ that accommodates one of the two central vertices of some copy of~$D$ can only accommodate $\eps n$ leaves of other copies.
If we wish to decompose a complete bipartite graph $K_{m,m}$ with copies of a tree $T_{n,n}$, then $m$ should be larger than $n$.
 See Section~\ref{final:treepackconj} for more discussion. 

%

Theorem~\ref{thm:packing_trees} can be applied to the setting of Conjecture~\ref{conj:ringel_bipartite}, as we can transform any tree $T:=T_{n+1}$ into a balanced tree $T':=T_{n+1, n+1}$ by adding an appropriate edge between  two copies of $T$. Theorem~\ref{thm:packing_trees} allows us to pack about $\sqrt n$ copies of $T'$ and thus about $2\sqrt n$ copies of $T$ into $K_{(1+\eps)n, (1+\eps)n}$, for any $\eps>0$ and large $n$. However, this is superseded by the result from~\cite{M23-large_rainbow} we mentioned above.

Returning to Conjecture~\ref{conj:tpc_balanced},
an immediate consequence of
Theorem~\ref{thm:packing_trees} is that, for any $\gamma > 0$ and sufficiently large $n$, any family of
 trees $T_{n,n},\dots,T_{n-\sqrt n+1,n-\sqrt n+1}$ packs into~$K_{(1+\gamma)n,(1+\gamma)n}$. Moreover, using heavy machinery from~\cite{KKOT19-blow_up}, we can embed a larger class of trees in the setting of Conjecture~\ref{conj:tpc_balanced}. Indeed, it is easy to check that along the lines of the proof of Theorem 8.5 from~\cite{KKOT19-blow_up}, one can deduce the following result from Theorem 8.1 of~\cite{KKOT19-blow_up}.
 \begin{theorem}\label{thm:gio}
  Suppose $0 < 1/n \ll \eps \ll\alpha, 1/\Delta$. Let $H_1,\dots,H_s$ be
  $2n$-vertex balanced bipartite graphs with $\Delta(H_i) \leq \Delta$, for
  each $i \in [s]$. Let $G$ be an $(\eps,d)$-super-regular balanced bipartite
  graph. If $\sum_{i=1}^{s}e(H_i) \leq (1-\alpha)e(G)$, then $H_1,\dots,H_s$
  pack into $G$.
\end{theorem}

%

An immediate consequence of Theorem~\ref{thm:gio} is the following.

\begin{corollary}
  Suppose $0 < 1/n \ll \alpha, 1/\Delta$. If
  $T_{\lceil \alpha n\rceil,\lceil\alpha n\rceil},\dots,T_{n,n}$ are balanced trees
  such that $T_{i,i}$ has $2i$ vertices, and $\Delta(T_{i,i}) \leq \Delta$,
  for each $\alpha n \leq i \leq n$,
  then $T_{\lceil \alpha n\rceil,\lceil\alpha n\rceil},\dots,T_{n,n}$ pack into $K_{n,n}$.
\end{corollary}
%
%
%
%
%
%

\subsection{A bipartite Koml\'os--S\'ark\"ozy--Szemer\'edi Theorem}

Our proof of Theorem~\ref{thm:packing_trees} is given in  Section~\ref{sec:thm4} and
relies on the following idea. We  embed the high
degree vertices of each tree into distinct vertices of $K_{n,n}$, and pack the rest of the trees into the remainder. For the packing, we
develop a tree embedding result for dense bipartite graphs, namely Theorem~\ref{thm:emb_rooted_tree} below.

We believe Theorem~\ref{thm:emb_rooted_tree} may be of independent interest and useful for other embedding problems with bipartite host graphs.
It can be viewed as a bipartite approximate version of the following well-known and widely used result.

\begin{theorem}[Koml\'os, S\'ark\"ozy and Szemer\'edi~\cite{KSS01-spanning_trees}]
  \label{thm:KSS}
For each~$\gamma > 0$ there are~$c, n_0 > 0$
such that for all~${n \geq n_0}$ every graph on $n$ vertices with minimum
degree at least~$(\frac 12 + \gamma)n$ contains a copy of every tree
on~$n$ vertices with maximum degree at most~$\frac{cn}{\log n}$.
\end{theorem}

A bipartite generalization of this theorem for graphs with
bounded maximum degree and sublinear bandwidth was proved by B\"{o}ttcher, Heinig, and Taraz in~\cite{BHT10-embedding}. Since every tree
with bounded maximum degree has sublinear bandwidth, their result
implies a bipartite version of Theorem~\ref{thm:KSS} for trees whose
maximum degree is bounded by a constant.
Our bipartite version of 
Theorem~\ref{thm:KSS}
 reads as follows.


\begin{theorem}
  \label{thm:emb_rooted_tree}
  For each $\gamma > 0$, there are $c, 
  n_0 > 0$ such that the following
  holds for every~${n \geq n_0}$. If $G=(A,B,E)$ is a balanced bipartite
  graph on $2n$ vertices with~${\delta(G) \geq (\frac{1}{2}+\gamma)n}$,
  and $T$ is a 
  balanced rooted tree on $2(1-\gamma)n$ vertices
  with $\Delta(T) \leq cn$, then~$T$ embeds in~$G$ with the root of~$T$
  embedded in~$A$.
\end{theorem}

Note that in Theorem~\ref{thm:emb_rooted_tree}, the bound on $\Delta(T)$ is
 better than in Theorem~\ref{thm:KSS} and in 
  the result from~\cite{BHT10-embedding}. However, our trees are only almost spanning, while in these other results the trees are spanning. Moreover, one does not expect to have a bound as in our theorem for spanning trees as we will explain in Section~\ref{sec:final}.

The proof of Theorem~\ref{thm:emb_rooted_tree} is given in
Section~\ref{sec:embed}, and relies on the regularity method and embedding results for trees which are discussed in  Section~\ref{sec:prelim}. Possible extensions of Theorem~\ref{thm:emb_rooted_tree} are discussed in  Section~\ref{sec:final}.

\section{Preliminaries}\label{sec:prelim}

\subsection{Regularity}

Let~$A,B$ be nonempty disjoint sets of vertices of a graph $G$.
Define their~\emph{density} as $d(A,B) := \frac{e(A,B)}{|A||B|}$.
For~$\eps >0$ we call~$A' \subseteq A$~\emph{$\eps$-significant}
if~${|A'| \geq \eps |A|}$, and analogously for $B'\subseteq B$. We say $(A,B)$
is \emph{$\eps$-regular} if $|d(A,B)-d(A',B')| \leq \eps$
for all $\eps$-significant subsets~$A' \subseteq A$, $B' \subseteq B$.
If furthermore $d(A,B)\ge d$ for some $d\ge 0$, we
call~$(A,B)$~\emph{$(\eps,d)$-regular}.
A vertex $v \in A$ is called
\emph{$\eps$-typical} to an $\eps$-significant set~$B'\subseteq B$ if
$\deg(v,B') \geq (d(A,B)-\eps)|B'|$, and use an analogous definition
for~$v \in B$. For $\eps>0$ we write $x=y\pm\eps$ if $x\in[y-\eps, y+\eps]$.

The next fact contains two well known properties of regular
pairs (see~\cite{KS96-regularity}).

\begin{fact}
  \label{fact:split_regular_pairs}
  Let $(A,B)$ be an $\eps$-regular pair of density $d$, and
  let $\eps, \delta>0$. Then the following hold:
  \begin{enumerate}[(a)]
    \item For each $\eps$-significant $B' \subseteq B$, all but at
      most $\eps|A|$ vertices from $A$ are $\eps$-typical to~$B'$.
    \item For all $\delta$-significant $A' \subseteq A$ and $B'\subseteq B$,
      the pair $(A',B')$ is~$\frac{2\eps}{\delta}$-regular with density
      $d \pm \eps$.
  \end{enumerate}
\end{fact}

We also need the following fact, which has been observed
before (see e.g.\cite{PS12-approximate_loebl}), but for completeness we include its short proof.

\begin{fact}\label{l:many-typical}
For all $i \in [s]$ let
$(X,Y_i)$ be $(\eps,d)$-regular and let $Y_i'\subseteq Y_i$
be an $\eps$-significant set. Then
at least $\bigl(1-\sqrt{\eps}\bigr)|X|$ vertices in~$X$ are~$\eps$-typical
to at least $(1-\sqrt{\eps})s$ sets $Y_i'$.
\end{fact}

\begin{proof}
  Suppose, for a contradiction, that the statement is false.
  Then more than~$\sqrt{\eps}|X|$~vertices in $X$
  are not $\eps$-typical to at least $\sqrt{\eps}s$ of the sets~$Y_i'$. So there are at least
  $\eps|X|s$ pairs $(v,i)$ such that $v\in X$ is not $\eps$-typical to $Y'_i$.
  Thus, there is $i \in [s]$ such that more than $\eps|X|$ vertices of $X$ are atypical to $Y'_i$,
  contradicting Fact~\ref{fact:split_regular_pairs}(a).
\end{proof}

Szemerédi's regularity lemma~\cite{S78-regular} states that every large graph
has a partition into a bounded number of vertex sets, most of which are pairwise $\eps$-regular.
We will need a version of this lemma for  bipartite
graphs $G=(V,W,E)$.
An \emph{$(\eps,d)$-regular partition of $G$} is a pair
$(\mathcal{X},\mathcal{Y})$ such that $\mathcal{X}$ is a partition of $V$, $\mathcal{Y}$ is a partition of $W$,
there exist $X_0 \in \mathcal{X}$ and $Y_0 \in \mathcal{Y}$
with $|X_0|, |Y_0| \leq \eps n$, and 
 for all $X \in \mathcal{X}\setminus\{X_0\}$ and $Y \in \mathcal{Y}\setminus \{Y_0\}$ it holds that
\begin{enumerate}
  \item $|X|=|Y|$,
  \item $X$ and $Y$ are independent sets, and
  \item $(X,Y)$ is $\eps$-regular with density either $d(X,Y) > d$ or
    $d(X,Y) = 0$.
\end{enumerate}
We often call the sets $X \in \mathcal{X}$, $Y\in \mathcal{Y}$ \emph{clusters}.
The \emph{$(\eps,d)$-reduced 
graph}~$R$ of~$G$ with
respect to~$(\mathcal{X},\mathcal{Y})$ is the graph on vertex
set~$(\mathcal{X}\setminus\{X_0\})\cup(\mathcal{Y}\setminus\{Y_0\})$ having an
edge between~${X \in \mathcal{X}}$
and~${Y \in \mathcal{Y}}$ whenever~${d(X,Y) > d}$.
We will use the following variant of Szemerédi's regularity lemma.

\begin{lemma}[Bipartite Regularity Lemma~\cite{BHT10-embedding}]
  \label{lem:regularity}
  For every $\eps > 0$ and $k_0 \in \NN$, there is a $K_0$ such
  that the following holds for every $d \in [0, 1]$ and $n \geq K_0$.
  For every balanced bipartite graph~$G$ on $2n$ vertices
  with~$\delta(G) \geq \lambda n$ for some $0 < \lambda < 1$, there exists
  a spanning subgraph $G' \subseteq G$,
  a graph $R$, and an integer $s$ with $k_0 \leq s \leq K_0$ with
  the following properties:
  \begin{enumerate}[(a)]
    \item $R$ is the $(\eps,d)$-reduced graph with respect to
      an~$(\eps,d)$-regular partition of~$G'$.
    \item $|V(R)| = 2s$ and $\delta(R) \geq (\lambda - (d+\eps))s$. \label{prop:reduced_graph}
  \end{enumerate}
\end{lemma}

\subsection{Tree partitioning and assigning results}

Our embedding strategy relies on the embedding of small trees into regular pairs.
The next result shows that any tree can be cut into small subtrees with few connecting vertices.
This type of result already appeared in the literature, probably first
in~\cite{AKS95-conjecture_loebl}. Here we use a version from~\cite{BPS19-degree_conditions}.
In what follows, let $r(T)$ denote the root of a rooted tree~$T$.

\begin{proposition}[Proposition 4.1 in~\cite{BPS19-degree_conditions}]
  \label{prop:split_trees}
    For each $\beta \in (0,1)$, each $t > \beta^{-1}$ and each rooted tree~$T$ with $t$ edges, there is  a set $S \subseteq V(T)$
    and a family $\calP$ of disjoint rooted trees such that
    \begin{enumerate}[(a)]
      \item $r(T) \in S$ and $|S| < \frac{1}{\beta} + 2$,
      \item $\calP$ consists of the components of $T-S$, with each
        tree $P\in\calP$ rooted at the vertex $r(P)$ closest to~$r(T)$ in $T$, and
      \item $|V(P)| \leq \beta t$ for each $P \in \calP$.
    \end{enumerate}
\end{proposition}
   The pair $(S,\mathcal P)$ from Proposition~\ref{prop:split_trees} is called a \emph{$\beta$-decomposition} of $T$.
    The vertices from $S$ are called \emph{seeds}.

Once the tree is decomposed, we will have to decide where to embed
each of the small trees in~$\mathcal P$. For this, we use the following result
from~\cite{SZ22-antidirected}.

\begin{lemma}[Lemma 3.5 in~\cite{SZ22-antidirected}]
  \label{lemma:camila}
  Let $m,s \in \NN, \mu > 0$ and let $(x_i,y_i)_{i \in I}\subseteq \NN^2$ be such that
  \begin{enumerate}[(a)]
    \item\label{camila1} $(1-\mu)\sum_{i \in I}x_i \leq \sum_{i \in I}y_i \leq (1+\mu)\sum_{i \in I}x_i$,
    \item\label{camila2} $x_i + y_i \leq \mu m$, for every $i \in I$, and
    \item\label{camila3} $\max\left\{\sum_{i \in I}x_i, \sum_{i \in I}y_i\right\} < (1-10\mu)ms$.
  \end{enumerate}
  Then there is a partition $\{J_1,\ldots,J_s\}$ of $I$  such that for each $i \in [s]$,
  \[
    \sum_{j \in J_i}x_j \leq (1-7\mu)m\;\;\text{ and }\;\;\sum_{j \in J_i}y_j \leq (1-7\mu)m.
  \]
\end{lemma}

\section{Embedding trees in dense balanced graphs: proof of Theorem~\ref{thm:emb_rooted_tree}}
\label{sec:embed}

 To improve readability, we omit all floors and ceilings.
  We may assume $\gamma < \frac 12$.
  Apply Lemma~\ref{lem:regularity} with~$\eps=\left(\frac{\gamma}{120}\right)^2$
  and $k_0=\frac1{\eps}$ to obtain~$K_0$. Set 
  \begin{equation*}\label{eq:n0def}
    c = \frac{\eps\gamma}{50K_0^2} \;\; \text{ and }\;\;n_0 = \frac{2K_0^4}{\eps}.
  \end{equation*}
  Let $n \geq n_0$ and let $G$ and $T$ be as in the statement of
  Theorem~\ref{thm:emb_rooted_tree}. We proceed in five steps:
  decomposition of~$G$, decomposition of~$T$ into $(S, \mathcal P)$, setting linking zones,
  assigning the trees in~$\mathcal P$ to clusters, and embedding $T$.

  \medskip\noindent
  \emph{Step 1.~Decomposition of $G$.}
  By Lemma~\ref{lem:regularity} with~$d = 5\sqrt{\eps}$ and $\lambda = \frac12+\gamma$,
  there are a spanning subgraph $G'\subseteq G$, a graph $R$, and an integer~$s$ with $k_0 \leq s \leq K_0$ such that
  $R$ is the $(\eps,d)$-reduced graph with respect to an $(\eps,d)$-regular
  partition~$(\mathcal{X},\mathcal{Y})$ of $G'$, with $\mathcal X=\{X_1,\dots,X_s\}$ and $\mathcal Y=\{Y_1,\dots,Y_s\}$ and
  \begin{equation}\label{mindegR}
    \delta(R)  \geq  \left(\frac{1}{2}+\frac{\gamma}{4}\right)s.
    \end{equation}
  A straightforward application of Hall's Theorem (see~\cite{hall})
 yields with a perfect matching~$M$ of~$R$. We can assume  $M$ pairs $X_i$ with $Y_i$ for each $i \in [s]$.

  \medskip\noindent
  \emph{Step 2.~Decomposition of $T$.} Set $t=|E(T)|$ and $\beta=\frac{\eps\gamma}{K_0^{4}}$.
  Proposition~\ref{prop:split_trees} provides a $\beta$-de\-com\-po\-si\-tion  $(S,\mathcal{P})$ of
  $T$.
  Define the set of \emph{linking vertices} as~${L(T)=\{r(P): P\in\mathcal{P}\}}\setminus\{r(T)\}$.
  Because the parent of each linking vertex is in $S$ and by our choice
  of~$c$, we have
  \begin{equation} \label{eq:n_links}
    | S \cup L(T)|\ \le \ |S| \cdot  \Delta(T) + 1 \ \leq \ \frac{4 K_0}{\eps}\cdot cn\ \leq \ \frac{\gamma}{8}|X_1|.
  \end{equation}

  \smallskip\noindent
  \emph{Step 3.~Setting linking zones.} We reserve a subset of each cluster $X_i$, $Y_i$ for~$L(T)$.
  To this end,  for every~$i \in [s]$ we arbitrarily partition~$X_i$~($Y_i$) into two sets~$X_{i,L}$
  and~$X_{i,P}$ ($Y_{i,L}$ and~$Y_{i,P}$)
  so that~${|X_{i,L}| \ = \ |Y_{i,L}| \ = \ \frac{\gamma}{4}|X_i|}$.
  We  call these subsets the \emph{L-slice} and the \emph{P-slice} of ~$X_i$~($Y_i$).  By Fact~\ref{fact:split_regular_pairs},
  each pair~$(X_{i,J},Y_{j,J'})$, with~$J,J' \in \{L,P\}$,
  is~$\frac{8\eps}{\gamma}$-regular with density exceeding~$\frac d2$.

  \medskip\noindent
  \emph{Step 4.~Assigning the vertices of $T$ to clusters.}
  We now decide which $P\in \calP$ will be
  embedded into each pair of clusters $(X_i,Y_i)$.
  Let $V_e$ (respectively,~$V_o$) be the set of all vertices of $T$ having even (resp.,~odd) distance
  to the root $r(T)$. Set $m:=|X_{1,P}|$ and for each $P \in \mathcal P$, set~$x_P:= |V_e \cap V(P)|$ and
  $y_P := |V_o \cap V(P)|$. As $T$ is 
  balanced,  conditions~\ref{camila1}
  and~\ref{camila3} of Lemma~\ref{lemma:camila} are satisfied with $\mu=c$,
  while~\ref{camila2} holds because~$|V(P)|\le\beta t$
  for each $P\in \mathcal P$.
  Thus, there is a partition $\{\mathcal P_1,\ldots,\mathcal P_s\}$ of~$\mathcal P$
  such that $\sum_{P\in \mathcal P_i} x_P \leq (1-7c)|X_{1,P}|$
  and~$\sum_{P\in \mathcal P_i}y_P \leq (1-7c)|X_{1,P}|$ for every $i \in [s]$.

  \medskip\noindent
  \emph{Step 5.~Embedding $T$.}
  We will successively embed all $v\in S$ and $P\in\calP$, starting with either the seed  $r(T)$ or the tree $P$ that contains  $r(T)$, and then always choosing a seed or a $P\in\calP$ which is adjacent to an already embedded vertex of $T$.
    We will ensure that each $v\in S$ will be embedded in a
  vertex $\phi(v)$ in the $P$-slice of some cluster such that if $\phi(v)\in X$ ($\phi(v)\in Y$), then
  \begin{equation}\label{goooood}
    \text{$\phi(v)$ is typical to $Y_{i, L}$ ($X_{i, L}$) for all but  at most
    $\sqrt{\eps}s$  indices $i$.}
  \end{equation}

  Assume we are about to embed a tree $P\in \mathcal P$,
  and let $i \in [s]$ be such that~$P\in\mathcal P_i$.
  If~$r(T)\in V(P)$,  we embed $r(T)$ into any vertex of $X_{i,P}$ that is typical towards
  $Y_{i,P}$, and then embed the remaining levels of $P$ into the $P$-slice of either
  $X_{i,P}$ or $Y_{i,P}$, using vertices that are typical with respect to the other set.
   Otherwise $r(P)$ is a linking vertex,
  and its parent $v$ was already embedded to a vertex $\phi(v)$ obeying~\eqref{goooood}.
  Say $\phi(v)\in X_{i'}$ (the case $\phi(v)\in Y_{i'}$ is analogous).
  By~\eqref{mindegR}, there is a cluster $Y_{\ell}$ such
  that $\phi(v)$ is typical to $Y_{\ell, L}$ and  $X_iY_{\ell}\in E(R)$.
  We embed $r(P)$ to a neighbor of $\phi(v)$ in  $Y_{\ell, L}$ that is typical
  to the (at this time) unused part of~$X_{i,P}$, and embed the remaining vertices of $P$ into $X_{i,P}\cup Y_{i,P}$,
  always choosing vertices that are typical to the currently unused
  part of the other set.

  Now assume we are about to embed a seed $v\in S$. Embed $v$ in a neighbor of the image of its parent (this condition is void if $v=r(T)$) and such that~\eqref{goooood} holds. This is possible by Fact~\ref{l:many-typical}.

  Note that by~\eqref{eq:n_links}, the $L$-slices are large enough to accommodate all seeds and
  linking vertices. Also, the $P$-slices are large enough
  to accommodate all other vertices, because of the properties of the partition $\{\mathcal P_1,\ldots,\mathcal P_s\}$.

\section{Packing large balanced trees: Proof of Theorem~\ref{thm:packing_trees}}
\label{sec:thm4}

We will use the following corollary of Theorem~\ref{thm:emb_rooted_tree}.
\begin{corollary}\label{cor:emb_rooted_tree}
  For each $\gamma > 0$, there are $c, 
    n_0 > 0$ such that the following holds
  for every~${n \geq n_0}$. If $G=(A,B,E)$ is a balanced bipartite graph on $2n$
  vertices with~${\delta(G) \geq (\frac{1}{2}+\gamma)n}$, and~$F=(A_F, B_F, E_F)$
  is a 
   balanced forest on at most  $2(1-\gamma)n$ vertices with
  $\Delta(F) \leq cn$, then $F$ embeds in~$G$, with~$A_F\subseteq A$.
\end{corollary}

\begin{proof}[Proof of Corollary~\ref{cor:emb_rooted_tree}]
Let $c, 
 n_0 > 0$ be given by Theorem~\ref{thm:emb_rooted_tree} for input $\gamma$. We can assume $cn_0\ge 2$ (otherwise simply choose $n_0$ sufficiently larger). 
We use induction on the number of components of $F$. If $F$ has only one component, we are done by Theorem~\ref{thm:emb_rooted_tree}.

  Otherwise,  $F$ must have two components $C_1$, $C_2$ such that there are two leaves $a\in V(C_1)\cap A_F$, $b\in V(C_2)\cap B_F$. Indeed, if this is not the case then either
  all leaves
  of $F$ are in the same component, which is absurd as all components have leaves, or one of $A_F$, $B_F$ only has vertices of degree at least two, which is impossible as $F$ is a balanced forest.
We add the edge~$ab$ to~$F$. The maximum degree of the resulting forest is $\max\{\Delta(F), 2\} \le cn$, and thus we can apply induction.
   \end{proof}

Corollary~\ref{cor:emb_rooted_tree} enables us to prove the following lemma.

\begin{lemma}\label{lem:packing_forests}
  For each $\gamma > 0$, there are $c, 
    n_0 > 0$ such
  that the following holds for every~$n \geq n_0$ and for
  all functions~$t(n),d(n)$ with~$c\cdot t(n)d(n)\le (\frac 12-\gamma)n$.
  If $\calF=\{F_1,\ldots,F_{t(n)}\}$ is a family of $t(n)$ 
   balanced
  forests, with $F_i = (A_i,B_i,E_i)$, each $F_i$ on at most $2(1-\gamma)n$
  vertices and of maximum degree at most~$c\cdot d(n)$,  
  then~$\calF$ packs into~$K_{n,n}$, with all $A_i$'s embedded in the same part
  of~$K_{n,n}$.
\end{lemma}

\begin{proof}
  Let~$c, 
    n_0$ be given by Corollary~\ref{cor:emb_rooted_tree}  for
  input~$\gamma$. We can assume that  $\gamma \leq \frac1{10}$ and~$n_0\ge 15$.  Let~$n \geq n_0$ and
  consider a family $\calF=\{F_1,\dots,F_{t(n)}\}$ as in the lemma.
  Denote one of the parts of $K_{n,n}$ by~$A$.
  Set~$G_1=K_{n,n}$. For $i=1,\ldots,t(n)$,
  we will find an embedding~$\phi$ of~$F_i=(A_i, B_i, E_i)$ into~$G_i$,
  with all vertices of $A_i$ embedded in~$A$, and then set~$G_{i+1}=G_i-\phi(E_i)$.
  This is possible by Theorem~\ref{thm:emb_rooted_tree},
  as for each $i \le t(n)$, we have
  $
   \delta(G_i) \geq n- (i-1)\,c\cdot d(n) > n- c\cdot t(n)d(n)
   >  (\frac12+\gamma)n.
  $
\end{proof}

Now we are ready to prove Theorem~\ref{thm:packing_trees}.

\begin{proof}[Proof of Theorem~\ref{thm:packing_trees}]
  Fix~$c,\, 
   n_0>0$ given by Lemma~\ref{lem:packing_forests}
  for input~$\gamma$. We may  assume that~\mbox{$c, \gamma \leq \frac 12$} and $n_0\ge (\frac8{c\gamma})^{2/\gamma}$.
  For~$n \geq n_0$, consider any family $\{T_1,\dots,T_t\}$
  of~$t \le n^{\frac12-\gamma}$ balanced rooted trees,
  each on at most~$2(1-\gamma)n$ vertices.
  Let~$A, B$ be the color classes of~$K_{n,n}$ and for $i=1,\ldots,t$
  let $A_i, B_i$ be the color class of $T_i$, with $A_i$ containing its root $r(T_i)$.
   Let $A'\subseteq A$ and $B'\subseteq B$ be sets of size
  $n':=\lfloor\gamma n\rfloor$.

    Let $H^A_i$ be the set of the~$\lfloor\frac{8\sqrt n}c\rfloor$ vertices of highest degrees in $A_i$, for $i=1,\ldots, t$. Define~$H^B_i$ analogously, and set $H_i:=H^A_i\cup H^B_i$. Since
  each tree $T_i$ has less than $2n$ edges, we know that each vertex in the balanced forest $F_i:=T_i - H_i$ has degree at most~$\frac{c\sqrt n}2\le c\sqrt{n-n'}$.
  Also,~$|V(F_i)|\le 2(1-\gamma)(n-n')$.
  We use Lemma~\ref{lem:packing_forests}, with~$t(n)=t$ and~$d(n)=\sqrt n$,
  to pack the forests~$F_1,\ldots, F_t$ into $K_{n,n} - (A'\cup B')$,
  with all the $A_i$'s embedded in~$A$.

We embed~$H^A=\bigcup_{1\le i\le t}H^A_i$ into
  $A'$, and $H^B=\bigcup_{1\le i\le t}H^B_i$  into~$B'$, which is possible as
  $$
    |H^A| =|H^B|\ < \ t\cdot \frac{8\sqrt n}c \ \leq \ \lfloor\gamma n\rfloor \ = \ |A'| \ = \ |B'|,
  $$
  by our choice of~$n_0$. This finishes the packing.
\end{proof}

\section{Final Remarks}
\label{sec:final}

\subsection{Open questions for tree-packing into bipartite graphs}\label{final:treepackconj}
As discussed in the introduction, if we wish to decompose a complete bipartite graph $K_{m,m}$ with copies of a tree $T_{n,n}$, then $m$ needs to  be larger than $n$, as in the following conjecture, which follows from  Conjecture~\ref{conj:ringel_bipartite}.

\begin{conjecture}
  \label{conj4}
  Any tree $T_{n,n}$ decomposes $K_{2n-1,2n-1}$.
\end{conjecture}

However, it may be possible that the  host graph can be somewhat smaller. We propose the following question.
\begin{question}
  \label{q1}
 What is the smallest $k$ such that every tree $T_{n,n}$ decomposes $K_{2n-1,k}$.
\end{question}

Clearly, $k$ needs to be at least $n$. Letting $D_{n,n}$ be the double-star with $n$ vertices in either partition class, one can prove by induction on $n$ that $K_{2n-1,n}$ decomposes into copies of  $D_{n,n}$. So one could think that $k=n$.
However, this is false for~$n = 3$ and paths. Indeed,
let $A$ and $B$ be the parts of $K_{5,3}$ such that $|A| = 5$, and
$|B| = 3$. Note that we cannot embed two vertices of degree two in a
vertex of $A$. Since each
path on six vertices has two vertices of degree two in both sides of
the bipartition, any decomposition of
 $K_{5,3}$ into three 6-vertex paths must
embed two vertices of degree two in the same vertex in $A$.
One can generalize this example for all $n\ge 3$.

Paths are not the only trees that force us to assume $k>n$ in Question~\ref{q1}. For instance, assume $n=2\ell+1$ is odd and
consider the tree $T$ that arises from taking two copies of $D_{\ell+1, \ell+1}$, deleting a leaf from each and
connecting the neighbors of these leaves by an edge. Each side of  $T$ contains two vertices of degree $\ell+1$.
So $T$ cannot decompose  $K_{4\ell+1, 2\ell+1}=K_{2n-1,n}$,
because the larger side would receive two of the high degree vertices,
which is impossible.

In a similar spirit as Question~\ref{q1}, we ask the following.

\begin{question}
  \label{q2}
  What is the smallest $k$ such that any family of $k$ balanced trees,
  each with $n$ vertices in either partition class, packs into $K_{2n-1,k}$\,?
\end{question}

Generalizing this even more, one could ask for the smallest $k$ such that
any family of balanced trees, each with at most $n$ vertices in either partition class, and with a total number of edges not exceeding $(2n-1)k$, packs into $K_{2n-1,k}$.

\subsection{An exact bipartite KSS theorem?}\label{final:KSS}
In
 Theorem~\ref{thm:emb_rooted_tree}
we have a $\gamma n$ slack on each side of the bipartition of the host graph.
 However, we believe the following direct analogue of
Theorem~\ref{thm:KSS} should be true in the bipartite setting.

\begin{conjecture}
  \label{conj:emb_rooted_tree}
  For each $\gamma > 0$, there are $c, n_0 > 0$ such that the following
  holds for every~${n \geq n_0}$. If $G=(A,B,E)$ is a balanced bipartite
  graph on $2n$ vertices with~${\delta(G) \geq (\frac{1}{2}+\gamma)n}$,
  and $T$ is a balanced rooted tree on $2n$ vertices
  with $\Delta(T) \leq \frac{cn}{\log n}$, then~$T$ embeds in~$G$ with
  the root of~$T$ embedded in~$A$.
\end{conjecture}

Note that   Conjecture~\ref{conj:emb_rooted_tree}  uses the same
bound $\Delta(T)\le cn/\log n$ as in Theorem~\ref{thm:KSS} (and not $cn$ as in
Theorem~\ref{thm:emb_rooted_tree}). This is necessary, as can be seen by considering the following adaptation of an example  from~\cite{KSS01-spanning_trees}. Consider $\lceil \alpha\log n\rceil$ stars whose sizes differ by at most one, where  $\alpha$ is  a constant.  Join a new vertex $r$ to each of the centers of the stars. Take two copies of the obtained tree and obtain a balanced tree $T$ by joining the two copies of $r$. The random balanced bipartite graph
with $p=0.9$ has w.h.p.\ minimum degree at least $0.8n$ but w.h.p.\
it does not contain~$T$ as a subgraph.

We were not able to  adapt previous
strategies for a proof of Conjecture~\ref{conj:emb_rooted_tree}. In
Koml\'os, S\'ark\"ozy, and Szemer\'edi's~\cite{KSS01-spanning_trees}
strategy, we could not find a way to distribute the leaves of the tree
into a star cover of the reduced graph.
An alternative route might be to use the absorption method,
as used, for example, by Kathapurkar and
Montgomery~\cite{KM22-spanning_trees}. Their strategy consists of splitting the tree $T$ into two subtrees one of which is small and serves for the absorption at the very end, while the other is large and is embedded with an approximate embedding result. As it is not possible to split any balanced tree into two balanced subtrees of the adequate sizes (for instance, for the tree from the previous paragraph this is not possible), we
believe that in order to use the strategy from~\cite{KM22-spanning_trees}, it would first be necessary to prove an unbalanced
variant of Theorem~\ref{thm:emb_rooted_tree}.

\section*{Acknowledgments}

The first author would like to thank Orlando Lee,
for pointing out Conjecture~\ref{conj:tpc_balanced},
and César Hernández-Vélez and José Coelho de Pina Jr.\ for
earlier discussions on Conjecture~\ref{conj:tpc_balanced}.


\bibliographystyle{abbrv}
\bibliography{packingTrees}

\end{document}